\documentclass[10pt]{amsart}

\usepackage{latexsym}
\usepackage{amssymb}
\usepackage{amsmath}

\newtheorem{theorem}{Theorem}[section]

\newtheorem{proposition}[theorem]{Proposition}

\theoremstyle{definition}
\newtheorem{definition}[theorem]{Definition}

\theoremstyle{remark}

\def\N{\mathbb{N}}
\def\Z{\mathbb{Z}}

\begin{document}

\title[An elementary proof of Jin's theorem with a bound]
{An elementary proof of Jin's theorem
\\
with a bound}

\author{Mauro Di Nasso}

\address{Dipartimento di Matematica\\
Universit\`a di Pisa, Italy}

\email{dinasso@dm.unipi.it}

\date{}

\begin{abstract}
\emph{We present a short proof of Jin's theorem which is entirely elementary,
in the sense that no use is made of nonstandard analysis,
ergodic theory, measure theory, ultrafilters, or other advanced tools.
The given proof provides the explicit bound $1/\text{BD}(A)\cdot\text{BD}(B)$ to
the number of shifts $A+B+m_i$ that are needed to cover a thick set.}
\end{abstract}



\maketitle

\section*{Introduction}

Many results in combinatorial number theory are about
structural properties of sets of integers that
only depends on their largeness as given by the density.
A beautiful result in this area was proved in 2000 by Renling Jin
with the tools of nonstandard analysis.

\smallskip
\begin{itemize}
\item
\textbf{Jin's theorem}: \emph{If $A,B\subseteq\N$ have
positive upper Banach density then their sumset
$A+B$ is piecewise syndetic.}
\end{itemize}

\noindent
(The upper Banach density is a refinement of the usual upper asymptotic
density; a set is piecewise syndetic if it has ``bounded gaps"
in suitable arbitrarily long intervals. See \S 1 below for precise definitions).
Many researchers showed interest in Jin's result but were not comfortable
with nonstandard analysis. In answer to that, a few years later Jin himself \cite{jin2}
directly translated his nonstandard proof into ``standard" terms,
but unfortunately in this way ``certain degree of intuition and motivation
are lost" [\emph{cit.}].

\smallskip
In 2006, with the use of ergodic theory, V. Bergelson, H. Furstenberg and
B. Weiss \cite{bfw} found a completely different proof of
that result, and improved on it by showing that the sumset
$A+B$ is in fact piecewise Bohr, a stronger property than
piecewise syndeticity. (This result was subsequently
stretched by J.T.~Griesmer \cite{gri} to cases where one of the
summands has null upper Banach density.)
Again by means of ergodic theory, V. Bergelson, M. Beiglb\"ock and A. Fish \cite{bbf}
elaborated a shorter proof and extended the validity of the theorem
to the general framework of countable amenable groups.
In 2010, M. Beiglb\"ock \cite{bei} found another proof by using
ultrafilters plus a bit of measure theory.
Recently, this author \cite{dn} applied nonstandard methods to
show several properties of difference sets, and gave yet another
different proof of Jin's result where an explicit bound to the number of
shifts of $A+B$ that are needed to cover arbitrarily large intervals is found.

\smallskip
In this paper we present a short proof of Jin's theorem in the strengthened
version mentioned above, which is entirely elementary and
hence easily accessible also for the non-specialists.
(Here, ``elementary" means that no use is made of
nonstandard analysis, measure theory and ergodic theory, ultrafilters,
or any other advanced tool.)
The underlying intuitions are close to some of the nonstandard arguments
in \cite{dn}, but of course formalization is different.
We paid attention to keep the exposition in this paper self-contained.

\medskip
\textbf{Notation:} By $\N=\{1,2,3,\ldots\}$ we denote the set
of \emph{positive} integers. If not specified otherwise, lower-case letters
$a, b, c, x, y, z, \ldots$ will denote integer numbers, and upper-case letters
$A, B, C, \ldots$ will denote sets of integers.
By writing $[a,b]$ we always denote intervals of integers, \emph{i.e.}
$[a,b]=\{x\in\Z\mid a\le x\le b\}$.

\bigskip
\section{Jin's theorem with a bound}

\medskip
Let us start by recalling three important structural notions for sets of integers.

\medskip
\begin{definition}
Let $A\subseteq\Z$ be a set of integers.

\smallskip
\begin{itemize}
\item
$A$ is \emph{thick} if it covers intervals of arbitrary
length, \emph{i.e.} if for every $k\in\N$ there exists an interval
$I=[y+1,y+k]$ of length $k$ such that $I\subseteq A$.

\smallskip
\item
$A$ is \emph{syndetic} if it has bounded gaps,
\emph{i.e.} if there exists $k$ such that $A\cap I\ne\emptyset$
for every interval $I$ of length $k$.

\smallskip
\item
$A$ is \emph{piecewise syndetic} if it covers arbitrarily
large intervals of a syndetic set, \emph{i.e.} if
$A=B\cap C$ where $B$ is thick and $C$ is syndetic.
\end{itemize}
\end{definition}

\medskip
Remark that thickness and syndeticity are dual notions, in the sense that
$A$ is thick if and only if its complement $A^c$ is not syndetic.

\smallskip
Recall the \emph{difference set} and the \emph{sumset} of
two sets of integers $A,B\subseteq\Z$:
$$A-B\ =\ \{a-b\mid a\in A,\ b\in B\}\ ;\quad
A+B\ =\ \{a+b\mid a\in A,\ b\in B\}.$$

With obvious notation, we shall simply write $A-z$ to indicate
the shift $A-\{z\}$.
It is easily shown that $A$ is syndetic if and only if $A+F=\Z$ for a suitable
finite set $F$; and that $A$ is piecewise syndetic if and only if
$A+F$ is thick for a suitable finite set $F$.

\smallskip
Let us now turn to concepts of largeness for sets of integers.
A familiar notion in number theory is that of \emph{upper asymptotic density}
$\overline{d}(A)$ of set of natural numbers $A\subseteq\N$, which is defined as the
limit superior of the relative densities of its initial segments:
$$\overline{d}(A)\ =\ \limsup_{n\to\infty}\frac{|A\cap[1,n]|}{n}.$$

The upper Banach density refines the density $\overline{d}$ to sets
of integers by considering arbitrary intervals instead of just initial intervals.

\medskip
\begin{definition}
The \emph{upper Banach density} of a set $A\subseteq\Z$ is defined as:
$$\text{BD}(A)\ =\
\lim_{n\to\infty}\left(\max_{x\in\Z}\frac{|A\cap[x+1,x+n]|}{n}\right).$$
\end{definition}

\medskip
One needs to check that such a limit always exists, and in fact
$\text{BD}(A)=\inf_{n\in\N}a_n/n$ where $a_n=\max_{x\in\Z}|A\cap[x+1,x+n]|$.
In consequence, if $\text{BD}(A)\ge\alpha$ then for every $n$.
Trivially $\text{BD}(A)\ge\overline{d}(A)$ for every $A\subseteq\N$.
The following properties, which directly follow from the definitions,
will be used in the sequel:

\medskip
\begin{itemize}
\item
$\text{BD}(A)=1$ if and only if $A$ is thick.

\smallskip
\item
The family of sets with null Banach density is closed under finite unions, \emph{i.e.}
if $\text{BD}(A_i)=0$ for $i=1,\ldots,k$, then $\text{BD}(A_1\cup\ldots\cup A_k)=0$.

\smallskip
\item
The Banach density is invariant under shifts, \emph{i.e.} $\text{BD}(A+x)=\text{BD}(A)$.
\end{itemize}

\medskip
Remark that the upper Banach density is not additive,
\emph{i.e.} there exist disjoint sets $A,B$ such that
$\text{BD}(A\cup B)<\text{BD}(A)+\text{BD}(B)$. However, for
families of shifts of a given set, additivity holds:

\medskip
\begin{itemize}
\item
If $A+x_i$ are pairwise disjoint sets for $i=1,\ldots,k$, then
$\text{BD}\left(\bigcup_{i=1}^k A+x_i\right)=k\cdot\text{BD}(A_i)$.
\end{itemize}

\medskip
A consequence of the above property is the following:

\medskip
\begin{itemize}
\item
If $\text{BD}(A)>1/2$ then every number is the difference of two elements
in $A$, \emph{i.e.} $A-A=\Z$.
\end{itemize}

\medskip
To see this, notice that for every $z$ one has $A\cap(A-z)\ne\emptyset$,
as otherwise $\text{BD}(A\cup(A-z))=2\cdot\text{BD}(A)>1$, a contradiction.
So $a=a'-z$ for suitable $a,a'\in A$, and hence $z\in A-A$.

\smallskip
An important general property of difference sets is given by the following
well-known result.

\medskip
\begin{proposition}
If $\text{BD}(A)>0$ then $A-A$ is syndetic.
\end{proposition}

\begin{proof}
If by contradiction $A-A$ was not syndetic, then
its complement $(A-A)^c$ would include a thick set $T$.
By the property of thickness, it is not hard to construct
an infinite set $X=\{x_1<x_2<\ldots\}$ such that
$X-X\subseteq T$. Since $(X-X)\cap(A-A)=\emptyset$, the sets in the family
$\{A-x_i\mid i=1\in\N\}$ are pairwise disjoint.
But this is not possible because if $k\in\N$ is such that $1/k<\text{BD}(A)$,
then one would get
$\text{BD}\left(\bigcup_{i=1}^k A-x_i\right)=
\sum_{i=1}^k\text{BD}(A-x_i)=k\cdot\text{BD}(A)>1$,
a contradiction.
\end{proof}

\medskip
Remark that the above property does \emph{not} extend to the general case $A-B$;
\emph{e.g.}, it is not hard to find thick sets $A,B,C$ such that
their complements $A^c,B^c,C^c$ are thick as well, and $A-B\subset C$.
However, $A-B$ is necessarily thick in case
the two sets are ``sufficiently dense". Precisely, the following holds:

\medskip
\begin{proposition}
Let $A\subseteq\Z$ be such that $\sup_{n\to\infty} n\cdot(a_n/n-\alpha)=+\infty$,
where $a_n=\max_{x\in\Z}|A\cap[x+1,x+n]|$.
If $\text{BD}(B)\ge 1-\alpha$, then $A-B$ is thick.
\end{proposition}

\begin{proof}
For every $k\in\N$, we show that an interval of length $k$ is included in $A-B$.
Let $N$ be such that $N\cdot(a_N/N-\alpha)>k$,
and pick an interval $I$ of length $N$ with $a_N=|A\cap I|$.
For every $i=1,\ldots,k$ we have that:
$$|(A-i)\cap I|\ \ge\ |A\cap I|-i\ >\ (\alpha\cdot N+k)-i\ \ge\
\alpha\cdot N.$$
Now recall that $\text{BD}(B)=\inf_{n\in\N}b_n/n$, so by the hypothesis
we can find an interval $J$ of length $N$ such that
$|B\cap J|\ge (1-\alpha)\cdot N$. Finally, pick $t$ such that
$t+J=I$. We claim that the interval $[t+1,t+k]\subseteq A-B$.
To show this, notice that for every $i=1,\ldots,k$ we have that
$$|(A-i)\cap I|+|(B+t)\cap I|\ =\ |(A-i)\cap I|+|B\cap J|\ >\
\alpha\cdot N + (1-\alpha)\cdot N\ =\ N\ =\ |I|.$$
So, $(A-i)\cap(B+t)\cap I\ne\emptyset$, and we can find
$a\in A$ and $b\in B$ such that $a-i=b+t$, and hence
$t+i\in A-B$.
\end{proof}

\medskip
Notice that $\text{BD}(A)>\alpha$ implies
$\sup_{n\to\infty} n\cdot(a_n/n-\alpha)=+\infty$,
which in turn implies $\text{BD}(A)\ge\alpha$; however,
neither implication can be reversed.
The fact that $A-B$ is thick whenever $\text{BD}(A)+\text{BD}(B)>1$
was first proved by M. Beiglb\"ock, V. Bergelson and A. Fish in \cite{bbf};
in fact, their proof actually shows the (slightly) stronger property
given in the previous proposition.

\smallskip
What presented so far is just a hint of the rich combinatorial structure of
sumsets and sets of differences, whose investigation seems still far from
being completed (see \emph{e.g.} the monographies \cite{tv,ru}).
In this area, a relevant contribution was given in 2000 by
Renling Jin. By working in the setting of hypernatural numbers
of nonstandard analysis, he showed that the appropriate structural
property to be considered for differences of dense sets is \emph{piecewise}
syndeticity.

\medskip
\begin{theorem}[Jin -- 2000]
If $A,B\subseteq\N$ have positive upper Banach density then their sumset
$A+B$ is piecewise syndetic.
\end{theorem}

\medskip
As mentioned in the introduction, Jin's theorem has been recently
re-proved by other means and with some improvements.
Here we shall present an elementary proof of the following
strengthening from \cite{dn}, where an explicit bound is given on the
number of shifts that are needed to cover a thick set.
(Recall that a set $C$ is piecewise syndetic if and only if
$C+F$ is thick for a suitable finite set $F$.)

\medskip
\begin{theorem}
[Jin -- with a bound]
Let $A,B\subseteq\Z$ have positive
upper Banach densities $\text{BD}(A)=\alpha$ and $\text{BD}(B)=\beta$, respectively.
Then there exists a finite set $F$ such that
$|F|\le 1/\alpha\beta$ and $(A-B)+F$ is thick.
\end{theorem}

\medskip
Notice that if $A,B\subseteq\N$ are sets of natural numbers then $A+B=A-(-B)$,
where $-B=\{-b\mid b\in B\}$. As trivially $\text{BD}(B)=\text{BD}(-B)$,
the above theorem immediately yields Jin's result about sumsets.

\bigskip
\section{The elementary proof}

\medskip
By the definition of upper Banach density, we can pick two sequences of integers
$\langle x_n\mid n\in\N\rangle$ and $\langle y_n\mid n\in\N\rangle$
such that if we put:

\medskip
\begin{itemize}
\item
$A_n=A\cap[x_n+1,x_n+n^2]$

\smallskip
\item
$B_n=B\cap[y_n+1,y_n+n]$
\end{itemize}

\smallskip
\noindent
then $\lim_{n\to\infty}|A_n|/{n^2}=\alpha$ and
$\lim_{n\to\infty}|B_n|/n=\beta$.
As the first step in the proof, for every $n$
we shall find a suitable shift of
$A_n$ that meets $B_n$ on a set whose relative density approaches $\alpha\beta$
as $n$ goes to infinity. To this end, we need the following lemma
from \cite{dn}. (In order to keep this paper self-contained, we
re-prove it here.)

\bigskip
\noindent
\textbf{Lemma 1.}
\emph{Let $N,n\in\N$. If $C\subseteq[1,N]$ and $D\subseteq[1,n]$
then there exists $z$ such that
$$\frac{|(C-z)\cap D|}{n}\ \ge\
\frac{|C|}{N}\cdot\frac{|D|}{n}-\frac{|D|}{N}.$$}

\begin{proof}
Let $\vartheta:\N\to\{0,1\}$ the characteristic function of $C$. Then:
\begin{eqnarray*}
\nonumber
\sum_{x=1}^N\left(\sum_{d\in D}\vartheta(x+d)\right) & = &
\sum_{d\in D}\left(\sum_{x=1}^N\vartheta(x+d)\right)\ =\
\sum_{d\in D}|C\cap[d+1,N]|
\\
\nonumber
{} & \ge &
\sum_{d\in D}(|C|-d)\ \ge\ |D|\cdot(|C|-n).
\end{eqnarray*}
By the \emph{pigeonhole principle}, there must
be at least one $z$ such that
$$\frac{1}{n}\sum_{d\in D}\vartheta(z+d)\ \ge\
\frac{1}{n}\cdot\frac{|D|\cdot(|C|-n)}{N}\ =\
\frac{|C|}{N}\cdot\frac{|D|}{n}-\frac{|D|}{N}.$$
Finally, notice that
$$\frac{|(C-z)\cap D|}{n}\ =\ \frac{|(D+z)\cap C|}{n}\ =\
\frac{1}{n}\sum_{d\in D}\vartheta(z+d).$$
\end{proof}

\medskip
For every $n$, apply the above lemma where $C=A_n-x_n\subseteq[1,n^2]$
and $D=B_n-y_n\subseteq[1,n]$. (Notice that $|C|=|A_n|$ and $|D|=|B_n|$.)
Then we can pick a suitable sequence
$\langle z_n\mid n\in\N\rangle$ such that
$$\frac{|(A_n-x_n-z_n)\cap(B_n-y_n)|}{n}\ \ge\
\frac{|A_n|}{n^2}\cdot\frac{|B_n|}{n}-\frac{|B_n|}{n^2}.$$
Now put:

\smallskip
\begin{itemize}
\item
$E_n=(A_n-x_n-z_n)\cap(B_n-y_n)\subseteq[1,n]$.
\end{itemize}

\medskip
Passing the above inequality to the limit, we obtain that
$$\lim_{n\to\infty}\frac{|E_n|}{n}\ \ge\ \alpha\beta.$$

In the second part of the proof we shall use the fact that
any sequence of subsets of $[1,n]$ whose relative
densities have a positive limit as $n$ approaches infinity, satisfies a relevant
combinatorial property about the corresponding difference sets.
Precisely:

\bigskip
\noindent
\textbf{Lemma 2.}
\emph{For $n\in\N$, let $E_n\subseteq[1,n]$.
If $\lim_{n\to\infty}|E_n|/n=\gamma>0$ then there exists
a finite set $F$ with $|F|\le 1/\gamma$
and such that the following property is satisfied:
\begin{itemize}
\item[$(\star)$]
$\text{BD}\left(\{n\mid [1,m]\subseteq (E_n-E_n)+F\}\right)>0$
for every $m\in\N$.
\end{itemize}}

\begin{proof}
We inductively define a finite increasing sequence
$\sigma=\langle m_i\mid i=1,\ldots,k\rangle$.
Set $m_1=0$. If property $(\star)$ is satisfied by $F=\{0\}$, then put
$\sigma=\langle m_1\rangle$, and stop. Otherwise,
let $m_2\in\N$ be the least counterexample. So,
$\Gamma_1=\{n\mid [1,m_2-1]\subseteq(E_n-E_n)+m_1\}$
has positive upper Banach density, but
$\Lambda_1=\{n\in\Gamma_1\mid m_2\in (E_n-E_n)+m_1\}$
has null upper Banach density. Notice that for
every $n\in\Gamma_1\setminus\Lambda_1$ one has
$(E_n+m_1)\cap(E_n+m_2)=\emptyset$.
If for every $m\in\N$ the set of all $n\in\Gamma_1$ such that
$[1,m]\subseteq\bigcup_{i=1}^2(E_n-E_n)+m_i$
has positive upper Banach density, then put $\sigma=\langle m_1,m_2\rangle$
and stop. Otherwise, let $m_3\in\N$ be the least counterexample. So,
the set $\Gamma_2=\{n\in\Gamma_1\mid [1,m_2-1]\subseteq\bigcup_{i=1}^2(E_n-E_n)+m_i\}$
has positive Banach density, but
$\Lambda_2=\{n\in\Gamma_2\mid m_3\in\bigcup_{i=1}^2(E_n-E_n)+m_i\}$
has null Banach density. Notice that for
every $n\in\Gamma_2\setminus\Lambda_2$ one has
$(E_n+m_i)\cap(E_n+m_3)=\emptyset$ for $i=1,2$.
Iterate this process. We claim that
we must stop at a step $k\le 1/\gamma$.
To see this, we show that whenever $m_1<\ldots<m_k$
are defined, one necessarily has $k\le 1/\gamma$.
This is trivial for $k=1$, so let us assume $k\ge 2$.

Notice that $\Lambda_1\cup\ldots\cup\Lambda_k$ has null Banach density,
and so $X=\Gamma_k\setminus(\Lambda_1\cup\ldots\cup\Lambda_k)$ has
positive Banach density (and hence it is infinite).
Since $X\subseteq\Gamma_i\setminus\Lambda_i$ for all $i$,
for every $N\in X$ the sets in the family $\{E_N+m_i\mid i=1,\ldots k\}$
are pairwise disjoint.
Now, every $E_N+m_i\subseteq[1,N+m_k]$, and so we
obtain the following inequality:
$$N+m_k\ \ge\ \left\vert\bigcup_{i=1}^k(E_N+m_i)\right\vert\ =\
\sum_{i=1}^k|E_N+m_i|\ =\ k\cdot|E_N|,$$
and hence
$$\frac{|E_N|}{N}\ \le\ \frac{1}{k}+\frac{m_k}{N}.$$
By taking limits as $N\in X$ approaches infinity,
one gets the desired inequality $\gamma\le 1/k$.
Finally observe that, by the definition of
$\sigma=\langle m_i\mid i=1,\ldots,k\rangle$,
for every $n\in\Gamma_k$ and for every $m\in\N$
we have the inclusion $[1,m]\subseteq\bigcup_{i=1}^k(E_n-E_n)+m_i$.
This shows that property $(\star)$ is fulfilled
by setting $F=\{m_1,\ldots,m_k\}$.
\end{proof}

\medskip
By the above Lemma where $\gamma=\alpha\beta>0$,
we can pick a finite set $F$ with $|F|\le 1/\alpha\beta$
and such that property $(\star)$ is satisfied by the sets
$$E_n\ =\ (A_n-x_n-z_n)\cap(B_n-y_n).$$
So, for every $m$ there exists $n$ (actually ``densely many" $n$) such that:
$$[1,m]\ \subseteq\ (E_n-E_n)+F\ \subseteq\
(A_n-x_n-z_n)-(B_n-y_n)+F\ \subseteq\ A-B+F-t_n,$$
and hence $[t_n+1,t_n+m]\subseteq A-B+F$, where
we denoted $t_n=x_n-y_n+z_n$. This shows that $A-B+F$ is thick,
and the proof is complete.

\bigskip
\section{Open problems}

\begin{enumerate}
\item
Lemma 2 states a much stronger property than
needed for the proof of the main theorem.
Can one derive a stronger result by a full use of that lemma?

\smallskip
\item
We saw in \S 1 that $A-B$ is thick whenever
$\text{BD}(A)+\text{BD}(B)>1$.
Can one combine this fact with similar arguments as the ones presented in this paper,
and prove interesting structural properties about $A-B$, $A-C$, $B-C$
under the assumption that $\text{BD}(A)+\text{BD}(B)+\text{BD}(C)>1$?
\end{enumerate}


\bigskip
\begin{thebibliography}{}

\bibitem{bbf}
\textsc{M.~Beiglb\"ock, V.~Bergelson and A.~Fish},
Sumset phenomenon in countable amenable groups,
\emph{Adv. Math.} \textbf{223}, pp.~416--432, 2010.

\bibitem{bei}
\textsc{M.~Beiglb\"ock},
An ultrafilter approach to Jin's theorem, \emph{Isreal J. Math.} \textbf{185},
pp.~369--374, 2011.

\bibitem{bfw}
\textsc{V.~Bergelson, H.~Furstenberg and B.~Weiss},
Piece-wise sets of integers and combinatorial
number theory, in \emph{Topics in Discrete Mathematics},
Algorithms Combin. \textbf{26}, Springer, Berlin, pp. 13-–37, 2006.

\bibitem{dn}
\textsc{M.~Di Nasso}, Embeddability properties of difference sets,
Arxiv:1201.5865 (2012), submitted.

\bibitem{gri}
\textsc{J.T.~Griesmer}, Sumsets of dense sets and sparse sets,
\emph{Isreal J. Math.} \textbf{190}, pp. 229--252, 2012.


\bibitem{jin1}
\textsc{R.~Jin}, The sumset phenomenon, \emph{Proc. Amer. Math. Soc.}, \textbf{130},
pp.~855--861, 2002.

\bibitem{jin2}
\textsc{R.~Jin}, Standardizing nonstandard methods for
upper Banach density problems, in
\emph{Unusual Applications of Number Theory} (M.~Nathanson ed.),
DIMACS Series, vol. 64, pp. 109--124, 2004.

\bibitem{ru}
\textsc{I.Z.~Rusza}, Sumsets and structure, part I of
\emph{Combinatorial Number Theory and Additive Group Theory}
(A.~Geroldinger and I.Z.~Ruzsa), Birkh\"auser, 2009.

\bibitem{tv}
\textsc{T.~Tao and V.H.~Vu}, \emph{Additive Combinatorics},
Cambridge University Press, Cambridge, 2006.
\end{thebibliography}
\end{document}